\documentclass[11pt]{amsart}
\usepackage[T1]{fontenc}
\usepackage{lmodern}
\usepackage{amsmath,amssymb,mathtools}
\usepackage{xcolor}
\usepackage[margin=1.05in]{geometry}
\usepackage[colorlinks=true,linkcolor=blue,citecolor=blue,urlcolor=blue]{hyperref}
\numberwithin{equation}{section}
\newtheorem{theorem}{Theorem}[section]

\newtheorem{lemma}[theorem]{Lemma}
\newtheorem{corollary}[theorem]{Corollary}
\theoremstyle{remark}
\newtheorem{remark}[theorem]{Remark}
\newcommand{\T}{\mathbb T}
\newcommand{\R}{\mathbb R}
\newcommand{\norm}[1]{\left\|#1\right\|}
\title[Sharp decay for nonlinearly damped NLS]{Asymptotic stability and sharp decay near homogeneous solutions of a nonlinearly damped Schr\"odinger equation on the torus}
\author{Boris Shakarov}

\begin{document}
\begin{abstract}
We study the one-dimensional periodic nonlinear Schr\"odinger equation with a power-type nonlinear damping, both in the focusing and the defocusing regime. For a sufficiently small damping parameter and initial data close in relative $H^1$ norm to a nonzero spatially homogeneous state, we prove global existence and asymptotic convergence, after a constant phase adjustment, to an explicit homogeneous solution. The \(L^2\) and \(H^1\) norms follow the explicit adiabatic decay law, with the same leading coefficient as the homogeneous solution. We obtain explicit convergence rates, including a \(t^{-1/2}\) bound for the relative spatial perturbation. For every nonconstant initial datum in a small neighborhood of the homogeneous solution, the kinetic energy has a strictly positive leading coefficient and decays one power of time faster than the squared \(L^2\) norm.
\end{abstract}
\maketitle

\section{Introduction}

The nonlinear Schr\"odinger equation 
\begin{equation}
 \label{eq:NLSNO}
 i\partial_tu-\partial_x^2u+\lambda|u|^{p}u = 0
\end{equation}
is an effective model for the evolution of slowly varying wave envelopes \cite{SS,Z68}. But in many relevant cases, dissipative effects, which are ignored in \eqref{eq:NLSNO}, can change the long-time evolution even when they are weak \cite{ASP,PSS}. For example, damping can prevent finite-time blow-up even when its coefficient is arbitrarily small \cite{ASP}. Thus, in this work we consider the damped nonlinear Schr\"odinger equation 
\begin{equation}\label{eq:nls}
 i\partial_tu-\partial_x^2u+\lambda|u|^{p}u
  +i\varepsilon|u|^{q}u=0,\qquad u(0)=u_0\in H^1(\T),
\end{equation}
where $\T=\R/\mathbb Z$, $p,q>0$, $\lambda\in\R$ and
$\varepsilon>0$ is a (small) damping parameter. Examples of the nonlinear damping include multi-photon absorption in nonlinear optics \cite{RKFM}, and three-body recombination in Bose-Einstein
condensates \cite{KMuS}.
Power-type damping provides a common setting for studying such
amplitude-dependent losses; see also \cite{ACS, ASP,PSS,Darwich}. For further background on damped nonlinear Schr\"odinger models, see for instance
\cite[Chapter~33]{Fibich15} and the references therein.

This model is dissipative as the $L^2$-norm of any non-trivial solution is decreasing. Indeed, denoting the mass of a solution by
\begin{equation}\label{eq:mass}
 m(t)=\norm{u(t)}_2
\end{equation}
one has
$$\frac{d}{dt}m^2=-2\varepsilon\norm u_{q+2}^{q+2}.$$
Moreover, H\"older's inequality provides the following upper bound on the decay rate
\begin{equation}\label{eq:upper}
 m(t)\le
 \frac{m(0)}
 {\bigl(1+q\varepsilon m(0)^{q}t\bigr)^{1/q}} := m_{\mathrm{ad}}(t), \qquad m_{\mathrm{ad}}(t)^{-q}=m(0)^{-q}+q\varepsilon t.
\end{equation}
We call $m_{\mathrm{ad}}$ the adiabatic decay law. It describes the decay of a uniform optical intensity or particle
density. Indeed, equality in \eqref{eq:upper} holds for
the exact spatially homogeneous solutions
\begin{equation}\label{eq:homogeneous}
 u_{\rm hom}(t,x)=m_{\mathrm{ad}}(t)e^{i(\theta+\Phi(t))},\qquad
 \Phi(t)=\lambda\int_0^t m_{\mathrm{ad}}(s)^{p}\,ds,
\end{equation}
where $\theta\in\R$. 

In this work, we investigate whether nearby non-constant solutions approach the homogeneous solution in \eqref{eq:homogeneous} relative to its decreasing amplitude, whether
their mass has the same leading decay coefficient, and how rapidly
their remaining spatial variation disappears.

The mass bound \eqref{eq:upper} alone does not distinguish these
questions from convergence to zero. In the usual orbital stability
formulation for a Hamiltonian standing wave, the reference profile
has a fixed nonzero size. But for a decaying reference
state, an absolute error may tend to zero while remaining comparable to, or becoming larger than, the reference amplitude. A relative estimate therefore carries information about the profile that an absolute decay estimate can miss.

A similar viewpoint is present in the linear modulation studies
of \cite{RKFM,CC}. Here we use the evolving norm $m(t)$ of the full
nonlinear solution and a minimizing phase to formulate relative
stability in $H^1$. Thus we introduce the decomposition

\begin{equation}\label{eq:decomposition}
 u(t,x)=e^{i\gamma(t)}\bigl(m(t)+\xi(t,x)\bigr),
 \qquad
 \xi(t,x):=e^{-i\gamma(t)}u(t,x)-m(t),
\end{equation}

where the phase is chosen by
\begin{equation}\label{eq:phasechoice}
 \norm{e^{-i\gamma(t)}u(t)-m(t)}_2^2
 =\inf_{\theta\in[0,2\pi)}
  \norm{e^{-i\theta}u(t)-m(t)}_2^2 = 
2m^2-2m \sup_{\theta\in[0,2\pi)} \Re\left(e^{-i\theta}\int_0^1u\,dx\right).
\end{equation}

When the spatial mean is nonzero, this phase is its argument, modulo
$2\pi$. The assumptions below ensure that the mean never vanishes.

Our main result establishes relative stability and sharp decay for
sufficiently small $\varepsilon$ and sufficiently small initial
relative perturbations. In the focusing case $\lambda <0$, we further impose the condition \eqref{eq:gap}. The number $4\pi^2$ is the first nonzero
eigenvalue of $-\partial_x^2$ on $\T$; this condition makes the
quadratic energy coercive in the directions transverse to the constant
orbit. It is automatic when $\lambda\ge0$.

Since the loss rate
is greater where the amplitude is larger, it is natural to ask whether
the evolution makes the spatial profile increasingly uniform relative
to its decreasing size.
We study this question on a periodic domain, where the long-time
behavior cannot be attributed to spatial spreading to infinity. Our result establishes that solutions starting close to the homogeneous one exist globally. Our theorem includes the focusing mass-supercritical regime
$p>4$ with weaker-power damping $q<p$, subject to
\eqref{eq:gap}. In the whole space, finite-time blow-up is known to persist in certain such regimes \cite{AnSh26Bu}. But most importantly, our result establishes relative homogenization near the homogeneous
family, with a clear separation between the decay rates of the
amplitude and the spatial variation.

\begin{theorem}\label{thm:main}
Let $p,q>0$, $\lambda\in\R$, and let $m(0)>0$ satisfy
\begin{equation}\label{eq:gap}
 4\pi^2+\lambda p m(0)^{p}>0.
\end{equation}
There exist
$\delta_*,\varepsilon_*,C>0$, depending only on $p,q,\lambda,m(0)$, such
that the following holds. Suppose
\begin{equation}
 \label{eq:paramcond}
  0<\varepsilon\le\varepsilon_*,
 \qquad \norm{u_0}_2=m(0),
 \qquad \norm{\xi(0)}_{H^1}\le\delta m(0),
 \qquad 0\le\delta\le\delta_*.
\end{equation}
Then \eqref{eq:nls} has a unique global solution in
$C([0,\infty);H^1(\T))$. Its spatial mean never vanishes, so
\eqref{eq:phasechoice} defines a representative of the minimizing phase $\gamma \in C^1([0,\infty),\R)$. Moreover the following properties are satisfied:

The mass follows closely the adiabatic law:
\begin{equation}\label{eq:masscorrection}
 0 \le m(t)^{-q}-m(0)^{-q}-q\varepsilon t \le C\delta^2m(0)^{-q} \log\bigl(1+q\varepsilon m(0)^{q}t\bigr).
\end{equation}
The $H^1$ norm has the adiabatic decay rate 
\begin{equation}\label{eq:massasymptotic}
 \lim_{t\to\infty}[q\varepsilon t]^{1/q} \norm{u(t)}_{2}=1, \qquad \lim_{t\to\infty}[q\varepsilon t]^{1/q} \norm{u(t)}_{H^1}=1.
\end{equation}
For the kinetic part, if $u_0$ is non-constant, then there exists $\ell>0$ such that
\begin{equation}\label{eq:kineticasymptotic}
 \frac{\norm{\partial_xu(t)}_2^2}{\varepsilon m(t)^{q+2}}
  \longrightarrow2\ell, \qquad \lim_{t\to\infty}[q\varepsilon t]^{(q+2)/q}
   \frac12\norm{\partial_xu(t)}_2^2=\ell\varepsilon.
\end{equation}
For constant data, the kinetic energy vanishes identically.

Finally, the solution converges to \eqref{eq:homogeneous} in the following sense
\begin{equation}\label{eq:relativebound}
 \frac{\norm{\xi(t)}_{H^1}}{m(t)}
 \le
 \frac{C\delta}
 {\bigl(1+q\varepsilon m(0)^{q}t\bigr)^{1/2}}.
\end{equation}
\end{theorem}

The central conclusion is that the spatial perturbation becomes small
relative to the background itself. Estimate \eqref{eq:relativebound}
gives a relative decay rate $t^{-1/2}$ and, for fixed
$\varepsilon>0$, the absolute bound
\[
 \norm{\xi(t)}_{H^1}
 \le C(\varepsilon)\delta(1+t)^{-1/q-1/2}.
\]
The relative exponent $1/2$ is independent of the conservative and
dissipative powers and of the sign of $\lambda$ even though the constants and the admissible neighborhood still depend on these parameters. 

The kinetic-energy limit makes this separation of scales sharp.
For every nonconstant initial datum in the stated neighborhood,
\eqref{eq:massasymptotic} and \eqref{eq:kineticasymptotic} imply
\[
 \lim_{t\to\infty}t\,
  \frac{\norm{\partial_xu(t)}_2^2}{m(t)^2}
  =\frac{2\ell}{q}>0.
\]
Thus the kinetic energy decays one power of time faster than the
squared $L^2$ norm, with a positive leading coefficient depending on
the solution. Since $\partial_xu=e^{i\gamma}\partial_x\xi$, this
also shows that the relative exponent $1/2$ cannot be improved for
such data. At the same time, $m/m_{\mathrm{ad}}\to1$: the leading
coefficient of the amplitude decay agrees with that of the 
homogeneous solution, despite the initial spatial perturbation.

We also compare
$\gamma$ with the reference phase $\Phi$ defined in \eqref{eq:homogeneous}. The following corollary
shows that their difference converges to a constant. The limiting
phase $\gamma_\infty$ in the statement below is explicitly defined in \eqref{eq:gammainfty}. Finally, we explicitly give a convergence rate towards the homogeneous solution.

\begin{corollary}\label{cor:phase}
Under the assumptions of Theorem~\ref{thm:main}, there exists
$\gamma_\infty\in\R$ and $C(\varepsilon) >0$, such that, as $t \to \infty$, 
\begin{equation}\label{eq:phaselimit}
|\gamma(t)-\Phi(t)-\gamma_\infty|
\le C(\varepsilon) \delta^2 \left[
(1+t)^{-p/q}\log(2+t)+(1+t)^{-1}
\right] \longrightarrow 0.
\end{equation}
Furthermore, 
\begin{multline}\label{eq:orbitlimit}
 \left\|u(t)-m_{\mathrm{ad}}(t)
   e^{i(\gamma_\infty+\Phi(t))}\right\|_{H^1}\\
 \quad \le C(\varepsilon)(1+t)^{-1/q}\Bigl[
  \delta(1+t)^{-1/2}
  +\delta^2(1+t)^{-p/q}\log(2+t)
  +\delta^2(1+t)^{-1}\bigl(1+\log(2+t)\bigr)\Bigr] \longrightarrow 0.
\end{multline}
\end{corollary}

\begin{remark}\label{rem:mean}
The constants in the theorem \ref{thm:main} and Corollary \ref{cor:phase} can be chosen uniformly when $m(0)$ varies
in a compact subset of $(0,\infty)$ on which \eqref{eq:gap} is satisfied. 
We also remark that if \(u_0\) has zero mean, then it lies outside the small neighborhood considered in this theorem. Indeed, for every \(\theta\),
\[
\|e^{-i\theta}u_0-m(0)\|_2^2
=2m(0)^2
-2m(0)\Re\left(e^{-i\theta}\int_0^1u_0\,dx\right)
=2m(0)^2.
\]
Consequently,
\[
\frac{\|\xi(0)\|_{H^1}}{m(0)}\ge\sqrt2.
\]
\end{remark}

\subsection{Nonzero carrier modes}
If $k\in2\pi\mathbb Z$ and $u$ solves \eqref{eq:nls}, direct
substitution shows that
\[
 U(t,x)=e^{i(kx+k^2t)}u(t,x+2kt)
\]
is also a periodic solution. The theorem therefore applies near any
pure Fourier mode after this change of coordinates. The faster decay
law concerns
\[
 \norm{(\partial_x-ik)U}_2=\norm{\partial_xu}_2.
\]
The ordinary derivative $\partial_xU$ also contains the carrier term
$ikU$, so its norm need not have this faster decay.

\subsection{Relation to previous work} We refer to the monograph \cite{SS} and references therein for early works on the dissipative NLS \eqref{eq:nls}.
Antonelli and Sparber \cite{ASP} and 
Antonelli, Carles and Sparber \cite{ACS} studied NLS with nonlinear damping, investigating the global well-posedness of the model. Rapti et al. \cite{RKFM} investigate linear stability of one specific single-Fourier-mode perturbation of \eqref{eq:homogeneous}. For this single-mode ansatz, they obtain explicit solutions in certain parameter regimes, with a relative mode amplitude that oscillates within an envelope decaying as $t^{-1/2}$. For more general parameter choices, their numerical simulations indicate that any initial growth of the perturbation is transient and that the dynamics eventually becomes stable. Theorem \ref{thm:main} establishes nonlinear and asymptotic stability in this direction for a full neighborhood of sufficiently small perturbations.

We mention early works of Weinstein on Lyapunov stability \cite{W86} for the ground state in the Hamiltonian case \eqref{eq:NLSNO} and the modulational stability in \cite{W85}, where \eqref{eq:NLSNO} is perturbed by a small term. In particular, applying the modulational argument, one can expect that an adaptation of \cite{W85} would yield (absolute) stability for times of order $1/\varepsilon$, hence, in general, not for all times.
Antonelli and Shakarov \cite{AS} developed a Lyapunov approach to stability under linear damping ($q = 0$ in \eqref{eq:nls}) for cnoidal waves of the cubic NLS on the torus. Their approach builds a perturbative Lyapunov, incorporating the first-order perturbation in the functional itself to control the remainder in absolute terms for all times. Linear damping alone does not imply relative homogenization.
Indeed, for $\lambda=0$ and linear damping, the squared
$L^2$ norm and the kinetic energy both decay as
$e^{-2\varepsilon t}$, so their ratio remains constant.The additional decay in Theorem~\ref{thm:main} describes the effect of
amplitude-dependent damping.

Decay in the whole space involves the interaction between spatial
spreading and dissipation. Many quantitative decay results therefore
use spatial localization assumptions on the initial data, expressed
through weighted norms or moment conditions. For example, Kita,
Miyazaki and Sato \cite{KMS,KMS26} establish $L^2$ decay estimates
for $H^1(\R^d)$ data satisfying $|x|u_0\in L^2(\R^d)$.
Cazenave, Han and Naumkin \cite{CHN21} obtain precise asymptotic
profiles for a class of localized oscillatory data, including an
exact $L^\infty$ decay law with a leading coefficient independent
of the initial datum. Their $L^2$ decay rate, however, depends on
the spatial decay of the initial profile. In general, for the same damping power, the $L^2$ decay rates are slower than the $t^{-1/q}$ law established
in Theorem~\ref{thm:main} reflecting the role of spatial spreading, which can weaken nonlinear dissipation at a given mass. For further whole-space decay estimates, optimality
results and asymptotic descriptions, see
\cite{Shim06,KS09,HLN16,KS22a,KS22b,GKS24,GKS25,KS23Bor}
and the references therein.

Damped NLS equations are also studied in the presence of an
additional forcing term. Forcing can balance dissipation and
sustain nonzero stationary states, as in the Lugiato--Lefever
equation; see \cite{HJPR23,HJPR24,Zum24,BR25,Alex26} and references
therein. G\"artner, Mandel and Reichel \cite{GMR22} include cubic
nonlinear damping in this model and prove convergence to
spatially constant states for sufficiently strong nonlinear
damping. In this case, as the limiting state is nonzero, convergence does not require a relative formulation.

We also mention that the classical Grillakis, Shatah and Strauss \cite{GSS} theory for Hamiltonian equations such as \eqref{eq:NLSNO} does not apply directly to this evolution. The coercivity
principle remains useful, but it must be combined with differential estimates along the dissipative flow. Here these estimates control the error relative to the vanishing amplitude and prove convergence of
the normalized profile.

\subsection{Notation and idea of the proof}
We use the norms
\[
 \norm f_k^k=\int_0^1|f(x)|^k\,dx,\qquad
 \norm f_{H^1}^2=\norm f_2^2+\norm{\partial_xf}_2^2.
\]
The real and imaginary parts are denoted by $\Re$ and $\Im$, and a dot
denotes a time derivative.

Constants may depend on the fixed parameters $p,q,\lambda,m(0)$;
constants denoted by $C(\varepsilon)$ may also depend on the fixed
positive damping parameter. 

The proof starts from the energy difference $\mathcal L$ defined in \eqref{eq:energy}. The spectral condition \eqref{eq:gap} (for $\lambda <0$) makes this
quantity comparable to $\norm\xi_{H^1}^2$ near the constant orbit.
Dividing by $m^2$ therefore measures the relative perturbation.
The leading dissipative term of $\mathcal L$ in its derivative controls only
$\norm{\partial_x\Re\xi}_2^2/m^2$. To recover control of the
imaginary component, we add the mixed term
\[
 \frac q2\varepsilon m^q
  \int_0^1\frac{\Re\xi}{m}\frac{\Im\xi}{m}\,dx.
\]
Its derivative contributes the difference between the real and
imaginary gradient energies. The remaining conservative contribution is controlled by a coefficient
proportional to $\varepsilon|\lambda|m^{p+q}$, whose time integral
is finite by the mass law. The corrected energy $\mathcal H$ defined
in \eqref{eq:corrected} consequently satisfies
\eqref{eq:mainerror}.

We first use this estimate to keep the relative perturbation small
and continue the solution globally. The resulting preliminary decay
makes the error terms in the logarithmic derivative of
$\mathcal H/(\varepsilon m^q)$ integrable. It follows that this
quotient converges to some $\ell>0$ for non-constant data. Together with
the mass asymptotic, this gives
\[
 \lim_{t\to\infty}t\varepsilon m(t)^q=\frac1q,
 \qquad
 \lim_{t\to\infty}t\mathcal H(t)=\frac\ell q.
\]
Since $\mathcal H$ is comparable to
$\norm\xi_{H^1}^2/m^2$, these relations explain why the relative
decay exponent is $1/2$ for every $p,q>0$. The power $q$ determines
the amplitude decay, while the leading damping coefficient
$q\varepsilon m^q$ in the corrected energy equation is asymptotically proportional to $1/t$. The kinetic-energy limit follows from comparison with
the gradient part of $\mathcal H$; the mass correction and phase
convergence follow from their respective evolution equations.
Because of the Sobolev embedding $H^1(\T) \hookrightarrow L^\infty(\T)$, the nonlinearities are expanded around $1$, with $\xi/m$ as the
perturbation. Since $1+\xi/m$ stays separated from zero for $\left| \xi/m \right|$ sufficiently small, these
expansions are valid for arbitrary real powers $p,q>0$.

The rest of this work is organized as follows. Section~\ref{sec:preliminary} establishes the mass, phase and coercivity
estimates. Section~\ref{sec:decay} proves relative stability and the
sharp decay bound. The final section \ref{sec:proof} derives the mass, kinetic-energy, and phase asymptotics, concluding the proof of Theorem \ref{thm:main} and Corollary \ref{cor:phase}.

\section{Preliminary estimates} \label{sec:preliminary}
\begin{lemma}\label{lem:mass}
As long as the solution exists, \eqref{eq:upper} holds and
\begin{equation}\label{eq:massderivative}
 \dot m=-\frac{\varepsilon}{m}\norm{m+\xi}_{q+2}^{q+2},\qquad
 \frac d{dt}m^{-q}
  =q\varepsilon\frac{\norm{m+\xi}_{q+2}^{q+2}}{m^{q+2}}.
\end{equation}
\end{lemma}
\begin{proof}
Take the imaginary part of the $L^2$ pairing of \eqref{eq:nls} with $u$ to obtain the 
\eqref{eq:massderivative}. Since $\norm u_{q+2}\ge\norm u_2=m$, integrating this inequality gives \eqref{eq:upper}.
\end{proof}

\begin{lemma}\label{lem:orthogonality}
The perturbation in \eqref{eq:decomposition} satisfies
\begin{equation}\label{eq:orthogonality}
 \int_0^1\Im\xi\,dx=0,\qquad
 2m\int_0^1\Re\xi\,dx=-\norm\xi_2^2.
\end{equation}
If $\norm\xi_{H^1}< \rho m$ for $0 \leq \rho <2$ then 
\begin{equation}\label{eq:poincare}
 \norm{\partial_x\xi}_2^2
 \le\norm\xi_{H^1}^2
 \le C\norm{\partial_x\xi}_2^2.
\end{equation}
\end{lemma}
\begin{proof}
The first identity in \eqref{eq:orthogonality} follows from \eqref{eq:phasechoice}, and the second follows
by expanding $\norm{m+\xi}_2^2=m^2$. To show the second inequality in \eqref{eq:poincare}, we exploit Poincar\'e's inequality and \eqref{eq:orthogonality} to get
\[
 \norm\xi_2^2
 \le\frac1{4\pi^2}\norm{\partial_x\xi}_2^2
  +\left|\int_0^1\xi\,dx\right|^2
 =\frac1{4\pi^2}\norm{\partial_x\xi}_2^2
  +\frac{\norm\xi_2^4}{4m^2}.
\]
The last term is absorbed into the left-hand side when $\rho < 2$.
\end{proof}

\begin{lemma}\label{lem:expansions}
Suppose $\norm\xi_{H^1}\le\rho m$, with $\rho>0$ sufficiently small.
For every $k>0$,
\begin{equation}\label{eq:moment}
 \left|
 \frac{\norm{m+\xi}_{k+2}^{k+2}}{m^{k+2}}
 -1-\frac{k(k+2)}2\frac{\norm{\Re\xi}_2^2}{m^2}
 \right|
 \le C_k\frac{\norm\xi_{H^1}^3}{m^3}.
\end{equation}
In particular,
\begin{equation}\label{eq:momentbound}
 1\le\frac{\norm{m+\xi}_{q+2}^{q+2}}{m^{q+2}}
  \le1+C\frac{\norm\xi_{H^1}^2}{m^2}.
\end{equation}
The phase satisfies
\begin{equation}\label{eq:phaseerror}
 \left|\dot\gamma-\lambda m^{p}\right|
 \le C\bigl(|\lambda|m^{p}+\varepsilon m^{q}\bigr)
  \frac{\norm\xi_{2}^2}{m^2}.
\end{equation}
\end{lemma}
\begin{proof}
The embedding $H^1(\T)\hookrightarrow L^\infty(\T)$ gives
$\norm\xi_\infty\le C\rho m$.
We choose $\rho$ so that $C\rho\le1/2$. Hence $|\xi(x)|\le m/2$
for every $x$, and 
\begin{equation}
 \label{eq:expansion1}
 \left|1 + \frac{\xi}{m}\right|^{k+2} = 1 + \frac{k+2}{m}\Re\xi + \frac{k+2}{2m^2}\vert{}\xi\vert{}^2 + \frac{k(k+2)}{2m^2}(\Re\xi)^2 + O\left(\frac{\vert{}\xi\vert{}^3}{m^3}\right)
\end{equation}
After integration, the linear and the $|\xi|^2$ term on the right-hand side of \eqref{eq:expansion1} cancel by \eqref{eq:orthogonality}, yielding \eqref{eq:moment} and the second inequality in \eqref{eq:momentbound}, while the first follows from $\norm{m+\xi}_{q+2}\ge\norm{m+\xi}_2=m$.

To derive the equation for $\dot\gamma$ we substitute $u=e^{i\gamma}(m+\xi)$ in
\eqref{eq:nls} to get
\begin{equation}
 \label{eq:xiequation}
  \partial_t\xi=-\dot m-i\dot\gamma(m+\xi)-i\partial_x^2\xi +i\lambda|m+\xi|^{p}(m+\xi) -\varepsilon|m+\xi|^{q}(m+\xi).
\end{equation}
Taking the imaginary parts gives
\[
 \partial_t\Im\xi= -\dot\gamma(m+\Re\xi)-\partial_x^2\Re\xi+\lambda|m+\xi|^{p}(m+\Re\xi) -\varepsilon|m+\xi|^{q}\Im\xi.
\]
Integrate over $[0,1]$. The left-hand side is zero by \eqref{eq:orthogonality}. The second-derivative term
integrates to zero by periodicity. Thus
\begin{equation}\label{eq:phaseexact}
 \left( m+\int_0^1\Re\xi\,dx \right) \dot\gamma= 
 \lambda\int_0^1|m+\xi|^{p}(m+ \Re \xi)\,dx
 -\varepsilon\int_0^1|m+\xi|^{q}(\Im \xi)\,dx.
\end{equation}
For $\rho\le1$, by \eqref{eq:orthogonality}, we have
\begin{equation}\label{eq:phasedenominator}
 m+\int_0^1\Re\xi\,dx
 =m-\frac{\norm\xi_2^2}{2m}
 \ge m\left(1-\frac{\rho^2}{2}\right)
 \ge\frac m2.
\end{equation}
By a Taylor expansion and exploiting \eqref{eq:orthogonality} we obtain
\[
\begin{split}
  \int_0^1|m+\xi|^{p}(m+\Re\xi)\,dx & =
 m^{p}\left(m+ (p+1) \int_0^1\Re\xi\,dx \right) + m^{p-1} O (\norm\xi_2^2) \\
 &= m^{p}\left(m+ \int_0^1\Re\xi\,dx \right) + m^{p-1} O (\norm\xi_2^2)
\end{split}
\]
For the second term on the right-hand side of \eqref{eq:phaseexact}, we observe that by \eqref{eq:orthogonality}, 
\[
 \left|\int_0^1|m+\xi|^{q}\Im\xi\,dx\right| =\left|\int_0^1\bigl(|m+\xi|^{q}-m^{q}\bigr)\Im\xi\,dx\right| \leq C m^{q-1}\norm\xi_2^2.
\]
This yields \eqref{eq:phaseerror}.
\end{proof}

Next, we introduce the energy and its difference from the energy of the constant state
\begin{equation}\label{eq:energy}
 E(u)=\frac12\norm{\partial_xu}_2^2+\frac\lambda{p+2}\norm u_{p+2}^{p+2},
 \qquad
 \mathcal L(t)=E(u(t))-\frac\lambda{p+2}m(t)^{p+2}.
\end{equation}

\begin{lemma}\label{lem:coercivity}
Assume \eqref{eq:gap}. For $\norm\xi_{H^1}\le\rho m$ with $\rho>0$
sufficiently small and $0<m\le m(0)$, there are constants $c,C>0$
such that
\begin{equation}\label{eq:Lcoercive}
 c\norm\xi_{H^1}^2\le\mathcal L\le C\norm\xi_{H^1}^2.
\end{equation}
\end{lemma}
\begin{proof}
Using \eqref{eq:moment} in the definition of $\mathcal L$ gives
\begin{equation}\label{eq:Lquadratic}
 2\mathcal L
 =\norm{\partial_x\xi}_2^2
 +\lambda p m^{p}\norm{\Re\xi}_2^2
 +O\bigl(|\lambda|m^{p-1}\norm\xi_{H^1}^3\bigr).
\end{equation}
If $\lambda\ge0$, the quadratic form controls
$\norm{\partial_x\xi}_2^2$. If $\lambda<0$, by \eqref{eq:orthogonality}, and Poincar\'e inequality we get
\[
 \norm{\Re\xi}_2^2
 \le\frac1{4\pi^2}\norm{\partial_x\Re\xi}_2^2
  +\frac{\norm\xi_2^4}{4m^2}.
\]
Since $m\le m(0)$, \eqref{eq:gap} and \eqref{eq:poincare} imply
\[
\begin{aligned}
\|\partial_x\xi\|_2^2+\lambda p m^p\|\Re\xi\|_2^2
\ge{}&
\|\partial_x\Im\xi\|_2^2
+\left(1+\frac{\lambda p m^p}{4\pi^2}\right)
     \|\partial_x\Re\xi\|_2^2 -\frac{|\lambda|p m^p}{4m^2}\|\xi\|_2^4 \ge c\norm\xi_{H^1}^2
\end{aligned}
\]
for $\rho$ sufficiently small. The remainder in \eqref{eq:Lquadratic} is
bounded by $C\rho\norm\xi_{H^1}^2$ and can also be absorbed.
\end{proof}

\section{Lyapunov functional and relative stability}\label{sec:decay}
Since $m(t)$ decays, relative stability requires control of the quantity
$\norm\xi_{H^1}^2/m^2$. We therefore consider $\mathcal L/m^2$. We define a \emph{corrected normalized energy}
\begin{equation}\label{eq:corrected}
 \mathcal H(t)=\frac{\mathcal L(t)}{m(t)^2}
 +\frac{q}{2}\varepsilon m(t)^{q}
  \int_0^1\frac{\Re\xi(t,x)}{m(t)}
    \frac{\Im\xi(t,x)}{m(t)}\,dx.
\end{equation}

\begin{lemma}\label{lem:corrected}
Assume \eqref{eq:gap}. There exist $\rho>0$ and $\varepsilon_*>0$
such that, whenever $0<m\le m(0)$,
$\norm\xi_{H^1}\le\rho m$ and $0<\varepsilon\le\varepsilon_*$,
\begin{equation}\label{eq:Hcoercive}
 c\frac{\norm\xi_{H^1}^2}{m^2}
 \le\mathcal H\le C\frac{\norm\xi_{H^1}^2}{m^2}.
\end{equation}
Moreover, the following estimates hold: 
\begin{equation}\label{eq:mainerror}
 \left|\dot{\mathcal H}+q\varepsilon m^{q}\mathcal H\right| \le C\varepsilon m^{q}
 \left(|\lambda|m^{p}+\varepsilon m^{q}
    +\frac{\norm\xi_{H^1}}m\right)\mathcal H,
\end{equation}
\begin{equation}\label{eq:Hgradient}
 \left|\mathcal H-\frac{\norm{\partial_x\xi}_2^2}{2m^2}\right|
 \le C\bigl(|\lambda|m^{p}+\varepsilon m^{q}\bigr)
     \frac{\norm{\partial_x\xi}_2^2}{m^2}.
\end{equation}
\end{lemma}

\begin{proof}
The bound
\[
\left|\mathcal H-\frac{\mathcal L}{m^2}\right|
\le\frac q4\varepsilon m^q\frac{\|\xi\|_2^2}{m^2}.
\]
together with \eqref{eq:Lcoercive} proves \eqref{eq:Hcoercive} once $\varepsilon_*m(0)^{q}$ is sufficiently
small. The expansion \eqref{eq:Lquadratic}, together with
\eqref{eq:poincare}, proves \eqref{eq:Hgradient}. It remains to prove \eqref{eq:mainerror}. We differentiate
$\mathcal H$ in time. We have
\begin{equation}\label{eq:normalizedLexact}
\begin{split}
 \frac d{dt}\left(\frac{\mathcal L}{m^2}\right) &=-\varepsilon m^{q}\int_0^1
 \left|1+\frac\xi m\right|^{q}
 \left(\left|\frac{\partial_x\xi}{m}\right|^2
  +q\left|\partial_x\left|1+\frac\xi m\right|\right|^2
 \right)\,dx\\
 &\quad-\varepsilon\lambda m^{p+q}
 \left(\frac{\norm{m+\xi}_{p+q+2}^{p+q+2}}{m^{p+q+2}}
  -\frac{\norm{m+\xi}_{q+2}^{q+2}}{m^{q+2}}\right)\\
 &\quad+2\varepsilon m^{q}
  \frac{\norm{m+\xi}_{q+2}^{q+2}}{m^{q+2}}\frac{\mathcal L}{m^2}.
\end{split}
\end{equation}
We now expand the three terms on the right-hand side of \eqref{eq:normalizedLexact}.
Since $|\xi|\le m/2$,
\[
 \left|\left|1+\frac\xi m\right|^{q}-1\right|
  \le C\frac{|\xi|}{m},\qquad
 \left|\partial_x\left|1+\frac\xi m\right|
    -\frac{\partial_x\Re\xi}{m}\right|
  \le C\frac{|\xi|}{m}\frac{|\partial_x\xi|}{m}.
\]
Thus the first integral differs from
\[
 - \varepsilon m^q \left( \frac{\norm{\partial_x\xi}_2^2
    +q\norm{\partial_x\Re\xi}_2^2}{m^2} \right)
\]
by a rest controlled by $C \varepsilon m^q \norm\xi_{H^1}^3/m^3$.
For the second term we use \eqref{eq:moment} with
$k=p+q$ and $k=q$. For the last term we use
\eqref{eq:momentbound} and \eqref{eq:Lquadratic} to get 
\[
\Bigg|
2\varepsilon m^q
 \frac{\|m+\xi\|_{q+2}^{q+2}}{m^{q+2}}
 \frac{\mathcal L}{m^2}
-\varepsilon m^q\frac{\|\partial_x\xi\|_2^2}{m^2}
-\varepsilon\lambda p m^{p+q}
    \frac{\|\Re\xi\|_2^2}{m^2}
\Bigg|
\le C\varepsilon m^q(1+|\lambda|m^p)
    \frac{\|\xi\|_{H^1}^3}{m^3}.
\]
The contributions involving $\norm{\partial_x\xi}_2^2/m^2$ cancel.
The coefficient of
$\varepsilon\lambda m^{p+q}\norm{\Re\xi}_2^2/m^2$ is 
\[
 -\frac{(p+q+2)(p+q)-q(q+2)}2+p
  =-\frac{p(p+2q)}2.
\]
Consequently,
\begin{equation}\label{eq:normalizedLderivative}
\left|
 \frac d{dt}\left(\frac{\mathcal L}{m^2}\right) +q\varepsilon m^{q}
     \frac{\norm{\partial_x\Re\xi}_2^2}{m^2} +\frac{p(p+2q)}2\varepsilon\lambda m^{p+q}
     \frac{\norm{\Re\xi}_2^2}{m^2}
\right| \le C\varepsilon m^{q}
  (1+|\lambda|m^{p})\frac{\norm\xi_{H^1}^3}{m^3}.
\end{equation}
To compute the time derivative of the second term on the right-hand side of \eqref{eq:corrected}, we exploit \eqref{eq:xiequation} to get
\begin{equation}\label{eq:relativeequation1}
 \begin{split}
 \partial_t\left(\frac\xi m\right)
 ={}&-i\partial_x^2\left(\frac\xi m\right)
  -i\dot\gamma\left(1+\frac\xi m\right)+i\lambda m^{p}
  \left|1+\frac\xi m\right|^{p}\left(1+\frac\xi m\right)\\
 &-\varepsilon m^{q}
  \left|1+\frac\xi m\right|^{q}\left(1+\frac\xi m\right)
  +\varepsilon m^{q}\frac{\norm{m+\xi}_{q+2}^{q+2}}{m^{q+2}}
     \left(1+\frac\xi m\right).
\end{split}
\end{equation}
Decomposing with Taylor expansion the two nonlinearities, we write
\begin{equation}\label{eq:relativeequation}
\begin{split}
 \partial_t\left(\frac\xi m\right)
 ={}&-i\partial_x^2\left(\frac\xi m\right)
  +i\lambda p m^{p}\frac{\Re\xi}{m} -q\varepsilon m^{q}\frac{\Re\xi}{m}
  +\mathcal R.
\end{split}
\end{equation}
where
\begin{equation}
 \label{eq:rest1}
\begin{aligned}
\mathcal R= &
-i(\dot\gamma-\lambda m^p)
  \left(1+\frac{\xi}{m}\right) +\varepsilon m^q
 \left(\frac{\|m+\xi\|_{q+2}^{q+2}}{m^{q+2}}-1\right)
  \left(1+\frac{\xi}{m}\right)\\
&+i\lambda m^p
 \left[
 \left|1+\frac{\xi}{m}\right|^p
  \left(1+\frac{\xi}{m}\right)
 -1-\frac{\xi}{m}-p\frac{\Re\xi}{m}
 \right]\\
&-\varepsilon m^q
 \left[
 \left|1+\frac{\xi}{m}\right|^q
  \left(1+\frac{\xi}{m}\right)
 -1-\frac{\xi}{m}-q\frac{\Re\xi}{m}
 \right].
\end{aligned}
\end{equation}
By \eqref{eq:moment} and \eqref{eq:phaseerror} we conclude that
\begin{equation}\label{eq:relativeerror}
 \norm{\mathcal R}_2
 \le C\bigl(|\lambda|m^{p}+\varepsilon m^{q}\bigr)
      \frac{\norm\xi_{H^1}^2}{m^2}.
\end{equation}
Now by taking the real and imaginary parts of \eqref{eq:relativeequation}
we have 
\[
\begin{aligned}
\frac{d}{dt}\int_0^1
 \frac{\Re\xi}{m}\frac{\Im\xi}{m}\,dx
=& -\frac{\|\partial_x\Im\xi\|_2^2}{m^2}
+\frac{\|\partial_x\Re\xi\|_2^2}{m^2} \\
 & +\lambda p m^p\frac{\|\Re\xi\|_2^2}{m^2} -q\varepsilon m^q
 \int_0^1\frac{\Re\xi}{m}\frac{\Im\xi}{m}\,dx
+\int_0^1\left(
 \Re\mathcal R\,\frac{\Im\xi}{m}
+\Im\mathcal R\,\frac{\Re\xi}{m}
\right)\,dx,
\end{aligned}
\]
which implies from \eqref{eq:massderivative} that
\begin{equation}\label{eq:mixedderivative}
\begin{split}
 \frac d{dt}\left[
  \frac{q}{2}\varepsilon m^{q}
  \int_0^1\frac{\Re\xi}{m}\frac{\Im\xi}{m}\,dx\right] & =\frac{q}{2}\varepsilon m^{q}
  \frac{\norm{\partial_x\Re\xi}_2^2
     -\norm{\partial_x\Im\xi}_2^2}{m^2}+\frac{qp}2\varepsilon\lambda m^{p+q}
     \frac{\norm{\Re\xi}_2^2}{m^2}\\
 &\qquad-\frac{q^2}{2}\varepsilon^2m^{2q}
  \left(1+\frac{\norm{m+\xi}_{q+2}^{q+2}}{m^{q+2}}\right)
  \int_0^1\frac{\Re\xi}{m}\frac{\Im\xi}{m}\,dx\\
 &\qquad+\frac{q}{2}\varepsilon m^{q}
  \int_0^1\left(\Re\mathcal R\,\frac{\Im\xi}{m}
      +\Im\mathcal R\,\frac{\Re\xi}{m}\right)\,dx.
\end{split}
\end{equation}
Combine \eqref{eq:normalizedLderivative} with
\eqref{eq:mixedderivative}. The second term on the left-hand side of \eqref{eq:normalizedLderivative} and the first on the right-hand side of \eqref{eq:mixedderivative} give 
\begin{equation}
 \label{eq:cancelation2}
 -q\varepsilon m^{q}
  \frac{\norm{\partial_x\Re\xi}_2^2}{m^2}+\frac{q}{2}\varepsilon m^{q} \frac{\norm{\partial_x\Re\xi}_2^2 -\norm{\partial_x\Im\xi}_2^2}{m^2} =-\frac{q}{2}\varepsilon m^{q} \frac{\norm{\partial_x\xi}_2^2}{m^2}.
\end{equation}
Now observe that 
\begin{multline}\label{eq:addedterm}
 q\varepsilon m^q\mathcal H
= \frac q2\varepsilon m^q
     \frac{\|\partial_x\xi\|_2^2}{m^2}
+\frac{qp}{2}\varepsilon\lambda m^{p+q}
     \frac{\|\Re\xi\|_2^2}{m^2} +\frac{q^2}{2}\varepsilon^2m^{2q}
  \int_0^1\frac{\Re\xi}{m}\frac{\Im\xi}{m}\,dx\\
+q\varepsilon m^q
 \left[
 \frac{\mathcal L}{m^2}
 -\frac{\|\partial_x\xi\|_2^2}{2m^2}
 -\frac{\lambda p}{2}m^p\frac{\|\Re\xi\|_2^2}{m^2}
 \right].
\end{multline}
After adding $q\varepsilon m^{q}\mathcal H$ to
$\dot{\mathcal H}$, by \eqref{eq:cancelation2}, the first term on the right-hand side of \eqref{eq:addedterm} cancels. For the second term on the right-hand side of \eqref{eq:addedterm}, we get by \eqref{eq:mixedderivative} and \eqref{eq:normalizedLderivative}
\[
\left(-\frac{p(p+2q)}2
 +\frac{qp}2+\frac{qp}2\right)\varepsilon\lambda m^{p+q}
       \frac{\norm{\Re\xi}_2^2}{m^2} = -\frac{p^2}{2}\varepsilon\lambda m^{p+q}
       \frac{\norm{\Re\xi}_2^2}{m^2}.
\]
Thus, it follows that 
\[
\begin{aligned}
\dot{\mathcal H}+q\varepsilon m^q\mathcal H
=&-\frac{p^2}{2}\varepsilon\lambda m^{p+q}
     \frac{\|\Re\xi\|_2^2}{m^2} -\frac{q^2}{2}\varepsilon^2m^{2q}
 \frac{\|m+\xi\|_{q+2}^{q+2}}{m^{q+2}}
 \int_0^1\frac{\Re\xi}{m}\frac{\Im\xi}{m}\,dx\\
&+\frac q2\varepsilon m^q
 \int_0^1\left(
 \Re\mathcal R\,\frac{\Im\xi}{m}
 +\Im\mathcal R\,\frac{\Re\xi}{m}
 \right)\,dx \\
 &+\left[
 \frac d{dt}\left(\frac{\mathcal L}{m^2}\right)
 +q\varepsilon m^q\frac{\|\partial_x\Re\xi\|_2^2}{m^2}
 +\frac{p(p+2q)}2\varepsilon\lambda m^{p+q}
     \frac{\|\Re\xi\|_2^2}{m^2}
 \right]\\
&+q\varepsilon m^q
 \left[
 \frac{\mathcal L}{m^2}
 -\frac{\|\partial_x\xi\|_2^2}{2m^2}
 -\frac{\lambda p}{2}m^p\frac{\|\Re\xi\|_2^2}{m^2}
 \right].
\end{aligned}
\]
Thus by \eqref{eq:momentbound}, \eqref{eq:relativeerror}, \eqref{eq:normalizedLderivative} and \eqref{eq:Lquadratic} (divided by $2m^2$), we conclude 
\[
 \begin{split}
 &\left|\dot{\mathcal H}+q\varepsilon m^{q}\mathcal H\right| \le C\varepsilon m^{q}
 \left(|\lambda|m^{p}+\varepsilon m^{q}
      +\frac{\norm\xi_{H^1}}m\right)
      \frac{\norm\xi_{H^1}^2}{m^2}.
 \end{split}
\]
Finally apply \eqref{eq:Hcoercive}. This proves \eqref{eq:mainerror}.
\end{proof}

\subsection{Keeping the relative perturbation small}
We now use \eqref{eq:mainerror} to prove global relative stability.
For every $u_0\in H^1(\T)$ and $p,q>0$, the standard local
well-posedness theory gives a unique solution
$u\in C([0,T];H^1(\T))$ for some $T>0$; see
\cite{ACS} for instance. Since the initial relative perturbation
is sufficiently small, the spatial mean is initially nonzero as highlighted in Remark \ref{rem:mean}.
By continuity, the minimizing phase and the perturbation are
defined on a sufficiently short interval, on which
$\|\xi(t)\|_{H^1}<\rho m(t)$. On this time interval, the solution stays separated from zero. Smooth approximation justifies the energy identities in integrated form.
In particular, $\mathcal H$ is locally absolutely continuous,
and \eqref{eq:mainerror} holds almost everywhere.

Moreover, the mass law implies
\[
 -\dot m\ge\varepsilon m^{q+1},\qquad
 -\frac d{dt}m^{p}\ge p\varepsilon m^{p+q},\qquad
 -\frac d{dt}(\varepsilon m^{q})
  \geq q \varepsilon^2m^{2q}.
\]
Integration gives two bounds we will use repeatedly:
\begin{equation}\label{eq:integrablecoeff}
 \varepsilon\int_0^t m(s)^{p+q}\,ds
  \le\frac{m(0)^{p}}{p},\qquad
 \varepsilon^2\int_0^t m(s)^{2q}\,ds
  \le\frac{\varepsilon m(0)^{q}}{q}.
\end{equation}
In particular, the conservative term on the right of
\eqref{eq:mainerror} has a finite time integral.

Now consider the maximal interval on which the solution
exists, its spatial mean is nonzero and $\norm\xi_{H^1}/m<\rho$.

Choose $\rho$ and $\varepsilon_*$ small enough that
$C(\rho+\varepsilon_*m(0)^{q})\le q/2$ in
\eqref{eq:mainerror}. That estimate then gives
\[
 \dot{\mathcal H}
 \le\left(-\frac{q}{2}\varepsilon m^{q}
     +C|\lambda|\varepsilon m^{p+q}\right)\mathcal H.
\]
Multiply by
\[
 \exp\left(\frac{q}{2}\varepsilon\int_0^t m(s)^{q}\,ds
  -C|\lambda|\varepsilon\int_0^t m(s)^{p+q}\,ds\right)
\]
and integrate. The first bound in \eqref{eq:integrablecoeff} shows that
\begin{equation}\label{eq:coarseH}
 \mathcal H(t)\le C\mathcal H(0)
  \exp\left(-\frac{q}{2}\varepsilon
      \int_0^t m(s)^{q}\,ds\right).
\end{equation}
By \eqref{eq:Hcoercive},
\begin{equation}\label{eq:coarsexi}
 \frac{\norm{\xi(t)}_{H^1}}{m(t)}
 \le C_0\frac{\norm{\xi(0)}_{H^1}}{m(0)}
  \exp\left(-\frac{q}{4}\varepsilon
      \int_0^t m(s)^{q}\,ds\right).
\end{equation}
Choose $\delta_*$ in \eqref{eq:paramcond} such that $C_0\delta_*<\rho$.
By \eqref{eq:coarsexi}, throughout this time interval of existence,
\[
 \frac{\norm{\xi(t)}_{H^1}}{m(t)}\le C_0\delta_*<\rho.
\]
We also need to exclude convergence to zero in finite time.
By \eqref{eq:momentbound}, we can choose $\rho$ so that
$1\le\norm{m+\xi}_{q+2}^{q+2}/m^{q+2}\le2$.
Integrating \eqref{eq:massderivative} yields
\begin{equation}\label{eq:masssandwich}
 \frac{m(0)}
 {\bigl(1+2q\varepsilon m(0)^{q}t\bigr)^{1/q}}
 \le m(t)\le
 \frac{m(0)}
 {\bigl(1+q\varepsilon m(0)^{q}t\bigr)^{1/q}}.
\end{equation}
Thus, on every finite interval, $m$ has a positive lower bound,
$|u|=|m+\xi|\ge m/2$ and $\norm u_{H^1}\le C m(0)$.
Moreover, \eqref{eq:orthogonality} gives
\[
 e^{-i\gamma}\int_0^1u\,dx
  =m-\frac{\norm\xi_2^2}{2m}\ge\frac m2>0.
\]
Thus the mean stays bounded away from zero on every finite interval,
and the minimizing phase in \eqref{eq:phasechoice} admits a $C^1$
real-valued representative that can be continued.
The $H^1$ bound and the positive lower bound on $|u|$ allow the local
solution to be continued past any finite endpoint. Together with
$\norm\xi_{H^1}/m\le C_0\delta_*<\rho$ and the nonvanishing mean,
this proves global existence and relative smallness for all $t\ge0$.

The lower bound in \eqref{eq:masssandwich} implies
\[
 \varepsilon\int_0^\infty m(t)^{q}\,dt=\infty,
\]
and \eqref{eq:coarsexi} consequently gives
$\norm\xi_{H^1}/m\to0$.

\subsection{Obtaining the exact decay rate}
The bound \eqref{eq:coarsexi} suffices to keep the solution close to
the constant orbit. To obtain the exponent $1/2$ in
\eqref{eq:relativebound}, we now integrate the full two-sided estimate
\eqref{eq:mainerror}.

We first record why its perturbation term is integrable. Differentiating
the exponential in \eqref{eq:coarsexi} shows directly that
\begin{equation}\label{eq:integrablexi}
 \varepsilon\int_0^\infty
  m(t)^{q}\frac{\norm{\xi(t)}_{H^1}}{m(t)}\,dt
 \le\frac{4C_0}{q}\frac{\norm{\xi(0)}_{H^1}}{m(0)}.
\end{equation}

Assume first that $u_0$ is nonconstant. Then by \eqref{eq:Hcoercive}
$\mathcal H(0)>0$, and the lower differential bound in
\eqref{eq:mainerror} implies $\mathcal H(t)>0$ at each finite time.
The main term in \eqref{eq:mainerror} suggests comparison with
$m^{q}$, because the mass law gives
\[
 \frac d{dt}\log m^{q}
 =-q\varepsilon m^{q}\frac{\norm{m+\xi}_{q+2}^{q+2}}{m^{q+2}}.
\]
We therefore divide $\mathcal H$ by $\varepsilon m^{q}$.
The constant factor $\varepsilon$ is included to match the
kinetic-energy limit in \eqref{eq:kineticasymptotic}.
Then 
\[
 \frac d{dt}\log\frac{\mathcal H}{\varepsilon m^{q}}
 =\frac{\dot{\mathcal H}
    +q\varepsilon m^{q}\mathcal H}{\mathcal H}+q\varepsilon m^{q}
   \left(\frac{\norm{m+\xi}_{q+2}^{q+2}}{m^{q+2}}-1\right).
\]
Use \eqref{eq:mainerror} for the first term and
\eqref{eq:momentbound} for the second to get 
\begin{equation}\label{eq:logratio}
 \left|\frac d{dt}\log\frac{\mathcal H}{\varepsilon m^{q}}\right|
 \le C\varepsilon m^{q}
  \left(|\lambda|m^{p}+\varepsilon m^{q}
       +\frac{\norm\xi_{H^1}}m\right).
\end{equation}
Its right-hand side has a finite integral by
\eqref{eq:integrablecoeff} and \eqref{eq:integrablexi}, with a bound
uniform for $\varepsilon\le\varepsilon_*$ and $\delta\le\delta_*$.
Hence the logarithm has a finite real limit, and
\begin{equation}\label{eq:Hlimit}
 \frac{\mathcal H(t)}{\varepsilon m(t)^{q}}
  \longrightarrow\ell\in(0,\infty).
\end{equation}
Integrating \eqref{eq:logratio} also gives, for all $t\ge0$,
\begin{equation}\label{eq:Hcomparison}
 C^{-1}\frac{\mathcal H(0)}{\varepsilon m(0)^{q}}
 \le\frac{\mathcal H(t)}{\varepsilon m(t)^{q}}
 \le C\frac{\mathcal H(0)}{\varepsilon m(0)^{q}}.
\end{equation}
Combining the upper bound with \eqref{eq:Hcoercive} and
\eqref{eq:upper}, we find
\[
\begin{split}
 \frac{\norm{\xi(t)}_{H^1}^2}{m(t)^2}
 &\le C\frac{\norm{\xi(0)}_{H^1}^2}{m(0)^2}
       \left(\frac{m(t)}{m(0)}\right)^{q} \le\frac{C\delta^2}{1+q\varepsilon m(0)^{q}t}.
\end{split}
\]
Taking square roots proves \eqref{eq:relativebound}. Finally, observe that for constant initial data, $\xi$ vanishes identically and the solution
is given by \eqref{eq:homogeneous}, so the same bound holds.

\section{Rest of the proofs of Theorem \ref{thm:main} and Corollary \ref{cor:phase}} \label{sec:proof}
We deduce the remaining parts of Theorem \ref{thm:main} from \eqref{eq:relativebound}, the mass law, and the limit \eqref{eq:Hlimit}.

\subsection{The mass and the $H^1$ norm}
Equations \eqref{eq:momentbound} and \eqref{eq:relativebound} give
\[
 0\le\frac{\norm{m+\xi}_{q+2}^{q+2}}{m^{q+2}}-1
 \le\frac{C\delta^2}{1+q\varepsilon m(0)^{q}t}.
\]
We substitute this into the second identity of
\eqref{eq:massderivative} and we integrate to get
\[
\begin{split}
 &0\le m(t)^{-q}-m(0)^{-q}-q\varepsilon t \le C\delta^2q\varepsilon\int_0^t
  \frac{ds}{1+q\varepsilon m(0)^{q}s} =C\delta^2m(0)^{-q}
  \log\bigl(1+q\varepsilon m(0)^{q}t\bigr).
\end{split}
\]
This is \eqref{eq:masscorrection}. Dividing \eqref{eq:masscorrection} by the second equation in \eqref{eq:upper} and using
$0\le1-(1+a)^{-1/q}\le a/q$ for $a\ge0$, we obtain
\begin{equation}\label{eq:amplituderatio}
 0\le1-\frac{m(t)}{m_{\mathrm{ad}}(t)}
 \le C\delta^2
 \frac{\log\bigl(1+q\varepsilon m(0)^{q}t\bigr)}
  {1+q\varepsilon m(0)^{q}t}.
\end{equation}
Thus $m/m_{\mathrm{ad}}\to1$, which proves the $L^2$ limit in
\eqref{eq:massasymptotic}. Since
\[
 \frac{\norm u_{H^1}^2}{m^2}
  =1+\frac{\norm{\partial_x\xi}_2^2}{m^2}\longrightarrow1
\]
by \eqref{eq:relativebound}, the $H^1$ limit follows as well.

\subsection{The positive kinetic-energy coefficient}
Suppose $u_0$ is nonconstant. By \eqref{eq:Hcoercive},
\eqref{eq:poincare} and the positivity of $\mathcal H(t)$,
$\norm{\partial_x\xi(t)}_2$ is nonzero at every finite time.
Equation \eqref{eq:Hgradient} gives
\[
 \left|\frac{m(t)^2\mathcal H(t)}
     {\norm{\partial_x\xi(t)}_2^2}-\frac12\right|
 \le C\bigl(|\lambda|m(t)^{p}+\varepsilon m(t)^{q}\bigr)
  \longrightarrow0.
\]
Here both terms tend to zero for every $p,q>0$.
Using \eqref{eq:Hlimit} and
$\norm{\partial_xu}_2=\norm{\partial_x\xi}_2$, we conclude that
\[
 \frac{\norm{\partial_xu(t)}_2^2}{\varepsilon m(t)^{q+2}}
 =\frac{\mathcal H(t)}{\varepsilon m(t)^{q}}
  \frac{\norm{\partial_x\xi(t)}_2^2}{m(t)^2\mathcal H(t)}
  \longrightarrow2\ell>0.
\]
This proves \eqref{eq:kineticasymptotic}. Combining it with the
mass asymptotic gives
\[
 [q\varepsilon t]^{(q+2)/q}
  \frac12\norm{\partial_xu(t)}_2^2\longrightarrow\varepsilon\ell.
\]
The kinetic energy is identically zero for constant data.
This completes the proof of Theorem~\ref{thm:main}.

\subsection{Convergence to the homogeneous solution}
We now prove Corollary~\ref{cor:phase}, using the reference phase
$\Phi$ already defined in \eqref{eq:homogeneous}.
We write
\[
 \dot\gamma-\dot\Phi
 =\bigl(\dot\gamma-\lambda m^{p}\bigr)
  +\lambda\bigl(m^{p}-m_{\mathrm{ad}}^{p}\bigr).
\]
We will show that both terms on the right have finite absolute
integrals. For the first term, observe that
\eqref{eq:phaseerror}, \eqref{eq:relativebound} and \eqref{eq:upper} imply
\[
 |\dot\gamma-\lambda m^{p}|
 \le C(\varepsilon) \delta^2 \left((1+t)^{-1-p/q}+(1+t)^{-2}\right).
\]
For the second term, \eqref{eq:masssandwich} ensures
$2^{-1/q}\le m/m_{\mathrm{ad}}\le1$.
Thus \eqref{eq:amplituderatio} gives
\begin{equation}
 \label{eq:phasediff}
 |m^p-m_{\mathrm{ad}}^p| \le C m_{\mathrm{ad}}^p
     \left|\frac m{m_{\mathrm{ad}}}-1\right| \le C(\varepsilon)\delta^2
     (1+t)^{-1-p/q}\log(2+t).
\end{equation}
In particular, both bounds are integrable and 
\begin{equation}\label{eq:diffPhases}
    |\dot\gamma-\dot\Phi|
\le C(\varepsilon)\delta^2
\left[(1+t)^{-1-p/q}\log(2+t)+(1+t)^{-2}\right].
\end{equation}
Thus we get $\gamma(t)-\Phi(t)\to\gamma_\infty$ where 
\begin{equation}
 \label{eq:gammainfty}
 \gamma_\infty :=\gamma(0)
  +\int_0^\infty(\dot\gamma(s)-\dot\Phi(s))\,ds.
\end{equation}
Moreover, integrating \eqref{eq:diffPhases} over $[t, \infty)$ and using \eqref{eq:gammainfty} gives 
\[
|\gamma(t)-\Phi(t)-\gamma_\infty|
\le C(\varepsilon)\delta^2
\left[(1+t)^{-p/q}\log(2+t)+(1+t)^{-1}\right].
\]
Finally, using \eqref{eq:decomposition} and
$m/m_{\mathrm{ad}}\le1$, we have
\begin{equation}\label{eq:weakconvergence}
 \left\|
\frac{u(t)}{m_{\mathrm{ad}}(t)e^{i(\gamma_\infty+\Phi(t))}}-1
\right\|_{H^1}
\le
\frac{\|\xi(t)\|_{H^1}}{m(t)}
+\left|\frac{m(t)}{m_{\mathrm{ad}}(t)}-1\right| +|\gamma(t)-\Phi(t)-\gamma_\infty| \to 0.
\end{equation}
Finally, to get the decay rate in \eqref{eq:orbitlimit}, we exploit \eqref{eq:relativebound}, \eqref{eq:amplituderatio}
and \eqref{eq:phaselimit} to bound the three terms on the right-hand side of \eqref{eq:weakconvergence}. Multiplying by $m_{\mathrm{ad}}(t)\le
C(\varepsilon)(1+t)^{-1/q}$ gives \eqref{eq:orbitlimit}.

For completeness, integration of the formula for $\Phi$ gives
\begin{equation}\label{eq:explicitphase}
 \Phi(t)=
 \begin{cases}
 \displaystyle
 \frac{\lambda m(0)^{p-q}}{\varepsilon(q-p)}
 \left[
 \bigl(1+q\varepsilon m(0)^{q}t\bigr)^{(q-p)/q}-1
 \right],&p\ne q,\\[6pt]
 \displaystyle
 \frac{\lambda}{q\varepsilon}
  \log\bigl(1+q\varepsilon m(0)^{q}t\bigr),&p=q.
 \end{cases}
\end{equation}
For $\lambda\ne0$, $\Phi$ has a finite limit if $p>q$,
a logarithmic leading term if $p=q$, and a power-law leading term
if $p<q$. In every case $\gamma-\Phi$ converges. If $\lambda=0$,
then $\Phi=0$ and $\gamma$ itself converges.

\bibliographystyle{abbrv}
\bibliography{DNLS_references}

\end{document}